\documentclass[11pt,a4paper]{article}

\usepackage[T1]{fontenc}
\usepackage[utf8]{inputenc}
\usepackage{amsmath,amssymb,amsthm}
\usepackage{booktabs}
\usepackage{microtype}
\usepackage[margin=1.1in]{geometry}
\usepackage[hidelinks]{hyperref}

\theoremstyle{plain}
\newtheorem{lemma}{Lemma}
\newtheorem{theorem}[lemma]{Theorem}
\newtheorem{proposition}[lemma]{Proposition}
\newtheorem{corollary}[lemma]{Corollary}
\theoremstyle{definition}
\newtheorem{remark}[lemma]{Remark}

\newcommand{\Cov}[3]{C(#1,#2,#3)}
\newcommand{\Covl}[3]{C_{\lambda}(#1,#2,#3)}
\newcommand{\Sch}[3]{L(#1,#2,#3)}
\newcommand{\blocks}{\mathcal{B}}
\newcommand{\pts}{X}
\newcommand{\gp}{G}
\newcommand{\link}{L}
\DeclareMathOperator{\ex}{ex}

\title{The covering number $\Cov{12}{6}{4}$ is $41$}
\author{Charlie Krug\\
  \small Independent researcher\\
  \small \href{mailto:charlie@charliekrug.com}{\texttt{charlie@charliekrug.com}}\\
  \small ORCID \href{https://orcid.org/0009-0001-6144-9421}{0009-0001-6144-9421}}
\date{July 25, 2026}

\begin{document}
\maketitle

\begin{abstract}
A $t$-$(v,k,\lambda)$ covering is a collection of $k$-subsets (blocks) of a $v$-set such that every
$t$-subset of points lies in at least $\lambda$ blocks; the covering number $\Covl{v}{k}{t}$ is the
least number of blocks in such a collection, and one writes $\Cov{v}{k}{t}$ when $\lambda=1$. The
recorded bounds for $\Cov{12}{6}{4}$ have been $40 \le \Cov{12}{6}{4} \le 41$. We show that no
$4$-$(12,6,1)$ covering with $40$ blocks exists, and hence that $\Cov{12}{6}{4}=41$. A counting
argument shows that in a hypothetical $40$-block covering every point lies in exactly $20$ blocks,
the link of every point is an optimal $3$-$(11,5,1)$ covering with a forced degree sequence, and
the six pairs of points of degree $10$ form a perfect matching;
an exhaustive case analysis over the orbits of a group of order $3840$, carried out by
satisfiability solving, then shows that no optimal $3$-$(11,5,1)$ covering occurs as such a link.
Each of the $81$ formulas in the primary proof has an unsatisfiability certificate checked by
\texttt{drat-trim} and by the formally verified checker \texttt{cake\_lpr}; two additional
cross-encoding certificates are checked by the same pipeline. The lower-bound argument uses no tabulated covering
number: its only numerical input, $\Cov{10}{4}{2} \ge 9$, is itself certified. As a by-product the
certificates yield a self-contained certified proof that the optimal $3$-$(11,5,1)$ covering is
unique up to isomorphism. Equivalently, the Tur\'an number $T(12,8,6)$ is $41$; the new value propagates to
improved lower bounds for $\Cov{13}{7}{5}$, $\Cov{14}{8}{6}$, $\Cov{15}{9}{7}$
and $\Cov{16}{10}{8}$.


\end{abstract}

\section{Introduction}

Let $v \ge k \ge t \ge 1$ and let $\pts$ be a set of $v$ points. A \emph{$t$-$(v,k,\lambda)$
covering} is a pair $(\pts,\blocks)$, where $\blocks$ is a collection of $k$-subsets of $\pts$
called \emph{blocks}, such that every $t$-subset of $\pts$ is contained in at least $\lambda$ blocks
of $\blocks$. The \emph{covering number} $\Covl{v}{k}{t}$ is the minimum number of blocks in such a
covering, and if $\lambda=1$ we write $\Cov{v}{k}{t}$. For a subset $Y \subseteq \pts$ we write
$b(Y)$ for the number of blocks of $\blocks$ containing $Y$, abbreviating $b(\{p\})$ to $b(p)$. We
follow the notation of Gordon and Stinson~\cite{HandbookCoverings} throughout.

The systematic study of covering numbers goes back to Erd\H{o}s and R\'enyi~\cite{ErdosRenyi1956}
and Erd\H{o}s and Hanani~\cite{ErdosHanani1963}; the asymptotic behaviour for fixed $k$ and $t$ was
settled by R\"odl~\cite{Rodl1985}. Exact values, by contrast, remain scarce. Most recorded upper
bounds come either from explicit constructions, such as those of Gordon, Kuperberg and
Patashnik~\cite{GKP1995}, or from stochastic search, as pioneered for covering designs by Nurmela
and \"Osterg\aa rd~\cite{NurmelaOstergard1993}; the standard reference for what is known in small
cases is the La Jolla Covering Repository~\cite{LJCR}.\footnote{The repository is in the process of
being superseded by the Covering Repository~\cite{CoveringRepository}, which imports its historical
data; the bounds quoted in this paper agree on both, and a snapshot of the repository data as of
2026-07-24 is pinned in the archived deposit.}

The basic lower bound is due to Sch\"onheim~\cite{Schonheim1964}: since the blocks through a fixed
point $p$, with $p$ deleted, form a $(t-1)$-$(v-1,k-1,\lambda)$ covering,
\begin{equation}\label{eq:schonheim}
  \Covl{v}{k}{t} \;\ge\; \left\lceil \frac{v}{k}\,\Covl{v-1}{k-1}{t-1} \right\rceil .
\end{equation}
Iterating \eqref{eq:schonheim} down to $t=1$ gives the Sch\"onheim bound $\Sch{v}{k}{t}$; in
practice one applies \eqref{eq:schonheim} with the best known value of $\Covl{v-1}{k-1}{t-1}$
substituted on the right, which is usually stronger. Improvements to \eqref{eq:schonheim} are
known: Hanani's refinement for $t=2$ (see \cite[Thm.~1.20]{HandbookCoverings}), the increment of
Mills and Mullin~\cite{MillsMullin} under a divisibility hypothesis, and the higher-incidence-matrix
bounds of Horsley~\cite{Horsley2014} and Horsley and Singh~\cite{HorsleySingh2017}, which improved
many hundreds of tabulated entries.

When this work began, the bounds recorded in the repository were
\begin{equation}\label{eq:known}
  40 \;\le\; \Cov{12}{6}{4} \;\le\; 41 .
\end{equation}
The lower bound is \eqref{eq:schonheim} applied with $\Cov{11}{5}{3}=20$, giving
$\lceil (12/6)\cdot 20 \rceil = 40$; the fully iterated Sch\"onheim bound is only
$\Sch{12}{6}{4}=36$. The upper bound is an explicit $41$-block covering recorded in the
repository~\cite{LJCR}. The value $\Cov{12}{6}{4}\le 41$ was already tabulated by Gordon,
Kuperberg and Patashnik~\cite[Table~4]{GKP1995}, whose table code attributes the construction to
the simulated-annealing searches of Nurmela and \"Osterg\aa rd~\cite{NurmelaOstergard1993}. The
repository's improvement history dates its own entry to the initial 1996 import of its database
and labels it only ``JCD article''; we verify the archived $41$-block design directly in
Section~\ref{sec:verification}.

None of the general improvements cited above closes the gap in \eqref{eq:known}; we verify this in
Remark~\ref{rem:novelty}. Our main result is that the upper bound in \eqref{eq:known} is exact.

\begin{theorem}\label{thm:main}
$\Cov{12}{6}{4} = 41$. That is, there is no $4$-$(12,6,1)$ covering with $40$ blocks.
\end{theorem}

The proof is computer-assisted, with a deliberately small mathematical core.
Section~\ref{sec:structure} shows by counting that a hypothetical $40$-block covering is completely
rigid: every point lies in exactly $20$ blocks, the link of every point is an \emph{optimal}
$3$-$(11,5,1)$ covering, every such optimal covering in turn has the forced degree sequence
$(10,9^{10})$, and the six pairs of points of degree $10$ form a perfect matching, which we normalise
once and for all. This reduces Theorem~\ref{thm:main} to a statement about $20$-subsets of the $462$
candidate link blocks. That statement is decided by an exhaustive case analysis over the orbits
of a group of order $3840$: Section~\ref{sec:symmetry} constructs the group,
Section~\ref{sec:encoding} encodes the cases as two families of propositional satisfiability
instances, and Section~\ref{sec:cases} resolves them. The analysis also yields, almost for free, a
self-contained certified proof of the uniqueness of the optimal $3$-$(11,5,1)$ covering
(Proposition~\ref{prop:unique}).
Section~\ref{sec:verification} describes the verification chain and states precisely what is and is
not certified; in particular, the lower bound depends on no tabulated covering number --- its only
numerical input, $\Cov{10}{4}{2}\ge9$, is proved by a certificate in the deposit.
Section~\ref{sec:concluding} records the Tur\'an reformulation and the bounds that improve as a
consequence.

A contemporaneous, methodologically distinct computer-assisted proof of
Theorem~\ref{thm:main} was announced by
D.~Bertram on 24~July~2026~\cite{Bertram2026}. That proof uses a different reduction, based on an
enumeration of $266$ link orbits; the present proof instead uses the forced perfect matching, a
fixed group of order $3840$, and certified non-extendability blockers.

\begin{remark}\label{rem:novelty}
Neither of the two strongest general improvements to \eqref{eq:schonheim} yields $41$ here. The
increment of Mills and Mullin~\cite{MillsMullin} requires
$v\,\Covl{v-1}{k-1}{t-1} \not\equiv 0 \pmod{k}$; at $(v,k,t)=(12,6,4)$ we have
$12\cdot 20 = 240 = 40\cdot 6$, so the hypothesis fails. (Its second hypothesis fails as well: one
would need $\Cov{11}{5}{3} = \binom{11}{r-1}\binom{5}{r-1}^{-1}\Cov{12-r}{6-r}{4-r}$ for some
$r \in \{2,3,4\}$, and the three values are $99/5$, $33/2$ and $33/2$, none equal to $20$.) For the
bounds of Horsley and Singh~\cite{HorsleySingh2017}, parametrised by $s \le \lfloor t/2 \rfloor$,
direct evaluation at $s=1$ and $s=2$ shows that the resulting inequalities do not reach $41$ for
these parameters; the evaluation is reproduced in the archived deposit.
\end{remark}

\section{Rigidity of a hypothetical 40-block covering}\label{sec:structure}

Everything in this section is elementary counting, except for one certified satisfiability fact
(Lemma~\ref{lem:c1042}), the seed of the entire lower-bound chain.

For a covering $(\pts,\blocks)$ and $p \in \pts$, the \emph{link} of $p$ is the collection
\[
  \blocks_p \;=\; \{\, B \setminus \{p\} \;:\; B \in \blocks,\ p \in B \,\} .
\]
Distinct blocks through $p$ give distinct members of $\blocks_p$, so $|\blocks_p| = b(p)$.

\begin{lemma}\label{lem:link}
Let $(\pts,\blocks)$ be a $t$-$(v,k,1)$ covering and $p \in \pts$. Then
$(\pts\setminus\{p\},\blocks_p)$ is a $(t-1)$-$(v-1,k-1,1)$ covering; consequently
$b(p) \ge \Cov{v-1}{k-1}{t-1}$, and \eqref{eq:schonheim} follows by summing over $p$.
\end{lemma}

\begin{proof}
Let $S$ be a $(t-1)$-subset of $\pts \setminus \{p\}$. Then $S \cup \{p\}$ is a $t$-subset of
$\pts$, so it lies in some block $B \in \blocks$; that block contains $p$, and
$S \subseteq B\setminus\{p\} \in \blocks_p$, a $(k-1)$-set. For the last assertion,
$k\,|\blocks| = \sum_{p} b(p) \ge v\,\Cov{v-1}{k-1}{t-1}$.
\end{proof}

\begin{lemma}\label{lem:c1042}
$\Cov{10}{4}{2} \ge 9$.
\end{lemma}

\begin{proof}
Suppose a $2$-$(10,4,1)$ covering with at most $8$ blocks exists. Every point has degree at least $3$, because each block
through a point covers only three of the nine pairs containing that point. On the other hand the
$8$-block upper limit gives at most $32$ point--block incidences, so some point has degree at most
$\lfloor 32/10\rfloor=3$. Thus some point has degree exactly $3$; relabel it as point $1$.
Each of the three blocks through point $1$ consists of point $1$ together with three of the other
nine points. Since all nine pairs containing point $1$ must be covered, these three three-point
sets cover $\{2,\dots,10\}$, and by cardinality they partition it. Relabelling $\{2,\dots,10\}$
therefore normalises the three blocks to
\[
 \{1,2,3,4\},\qquad \{1,5,6,7\},\qquad \{1,8,9,10\}.
\]
The certificate formula has one variable per $4$-subset of $\{1,\dots,10\}$, one coverage clause
per pair, a cardinality constraint allowing at most $8$ blocks, positive units for these three
normalised blocks, and negative units for every other block through point $1$. It is
unsatisfiable, with a certificate checked as described in Section~\ref{sec:verification}.
\end{proof}

The value $\Cov{10}{4}{2}=9$ is classical; Lemma~\ref{lem:c1042} re-proves the direction we use, so
that no external table entry enters the chain of reasoning.

\begin{proposition}\label{prop:c1153}
$\Cov{11}{5}{3} = 20$.
\end{proposition}

\begin{proof}
By Lemma~\ref{lem:link} and \eqref{eq:schonheim} with Lemma~\ref{lem:c1042},
$\Cov{11}{5}{3} \ge \lceil \tfrac{11}{5} \cdot 9 \rceil = \lceil 19.8 \rceil = 20$. Conversely, an
explicit $3$-$(11,5,1)$ covering with $20$ blocks is archived in the deposit and machine-verified;
it is the design $E$ of Remark~\ref{rem:repair} below.
\end{proof}

The value $\Cov{11}{5}{3}=20$ is due to Mills~\cite{MillsC1153};
Proposition~\ref{prop:c1153} gives a short, independently checkable rederivation.

\begin{lemma}\label{lem:degrees}
Every $3$-$(11,5,1)$ covering with $20$ blocks has degree sequence $(10,9^{10})$: exactly one point
of degree $10$ and ten points of degree $9$.
\end{lemma}

\begin{proof}
By Lemma~\ref{lem:link} applied to the covering itself, every point $q$ satisfies
$b(q) \ge \Cov{10}{4}{2} \ge 9$. Counting incidences, $\sum_q b(q) = 20 \cdot 5 = 100 = 9\cdot11+1$.
Eleven summands, each at least $9$, totalling $9\cdot 11 + 1$, must be ten nines and one ten.
\end{proof}

\begin{lemma}\label{lem:forcing}
Let $(\pts,\blocks)$ be a $4$-$(12,6,1)$ covering with $|\blocks| = 40$. Then $b(p) = 20$ for every
$p \in \pts$, and the link of every point is an optimal $3$-$(11,5,1)$ covering with $20$ blocks
and degree sequence $(10,9^{10})$.
\end{lemma}

\begin{proof}
By Lemma~\ref{lem:link} and Proposition~\ref{prop:c1153}, $b(p) \ge \Cov{11}{5}{3} = 20$ for every
$p$. Counting incidences,
$\sum_{p} b(p) = 40 \cdot 6 = 240 = 12 \cdot 20$, so each of the twelve summands equals $20$
exactly. Each link is then a $3$-$(11,5,1)$ covering (Lemma~\ref{lem:link}) with
$20 = \Cov{11}{5}{3}$ blocks, and Lemma~\ref{lem:degrees} gives its degree sequence.
\end{proof}

\begin{lemma}\label{lem:matching}
In a $4$-$(12,6,1)$ covering with $40$ blocks, every pair of points has degree $9$ or $10$.
The six degree-$10$ pairs form a perfect matching of the point set.
\end{lemma}

\begin{proof}
For distinct points $p,q$, the blocks containing $\{p,q\}$, with that pair deleted, form a
$2$-$(10,4,1)$ covering. Hence $b(pq)\ge\Cov{10}{4}{2}\ge9$ by
Lemma~\ref{lem:c1042}. Fix $p$. Each block through $p$ contains five pairs incident with $p$, so
\[
  \sum_{q\ne p} b(pq)=5b(p)=100
\]
by Lemma~\ref{lem:forcing}. The sum has eleven terms, each at least $9$; consequently exactly one
term is $10$ and the other ten are $9$. Thus every point has a unique degree-$10$ partner.
Because pair degree is symmetric, these partner relations are six disjoint pairs covering all
twelve points.
\end{proof}

Lemma~\ref{lem:forcing} is the reason these parameters are both tractable and delicate: the bound
\eqref{eq:known} is attained with no slack whatsoever. A single point of degree $21$ would force
another of degree $19$, which is impossible. The entire problem is thereby transported to the
$11$-point link structure.

We fix a point $p_0 \in \pts$, the \emph{root point}. By Lemma~\ref{lem:matching} we may label the
perfect matching of degree-$10$ pairs as
\[
 \{p_0,1\},\ \{2,3\},\ \{4,5\},\ \{6,7\},\ \{8,9\},\ \{10,11\}.
\]
This labelling is without loss of generality for the full $12$-point covering. In particular, the unique point of
degree $10$ in the link $\blocks_{p_0}$ is point $1$. By Lemma~\ref{lem:forcing},
$\blocks_{p_0}$ is then a set of
$20$ of the $\binom{11}{5} = 462$ possible $5$-subsets of $\{1,\dots,11\}$, covering all
$\binom{11}{3} = 165$ triples, with $b(1)=10$ and $b(q)=9$ for $q \ge 2$.

Optimal $3$-$(11,5,1)$ coverings exist (Proposition~\ref{prop:c1153}), so
Lemma~\ref{lem:forcing} alone does not finish the argument. What must be shown is that no such
covering \emph{extends}: none of them arises as the link of a point in a $40$-block covering. The
remaining sections do this.

\section{Symmetry}\label{sec:symmetry}

The case analysis of Section~\ref{sec:cases} branches over the orbits of a fixed group of
symmetries of candidate links; this section constructs the group and computes its orbits.

Retain the five pairs $\{2,3\},\{4,5\},\{6,7\},\{8,9\},\{10,11\}$ of the normalisation above, and let
$\gp$ be the group of permutations of $\{1,\dots,11\}$ generated by the transpositions within the
pairs together with the permutations of the five pairs as blocks. Then $\gp \cong C_2 \wr S_5$,
\begin{equation}\label{eq:grouporder}
  |\gp| \;=\; 2^5 \cdot 5! \;=\; 3840 ,
\end{equation}
and every element of $\gp$ fixes the point $1$. Thus $\gp$ is exactly the stabiliser, on the
$11$ link points, of the normalised perfect matching. The group was rebuilt from this description,
and the $3840$ listed elements were verified to be pairwise distinct bijections fixing the point
$1$ and closed under composition and inversion.

\begin{proposition}\label{prop:orbits}
Acting on the $462$ five-subsets of $\{1,\dots,11\}$, the group $\gp$ has exactly six orbits.
Writing the \emph{type} of a $5$-set $S$ as the multiset of its intersection sizes with the five
pairs, together with whether $1 \in S$, the orbits and their sizes are:
\begin{center}
\begin{tabular}{llll}
\toprule
 & type (pairs) & size & orbit \\
\midrule
$1 \in S$: & $1{+}1{+}1{+}1$ & $\binom{5}{4}\,2^4 = 80$ & $(=\mathcal{O}_0)$\\
           & $2{+}1{+}1$     & $5\binom{4}{2}2^2 = 120$ & $(=\mathcal{O}_1)$\\
           & $2{+}2$         & $\binom{5}{2} = 10$      & $(=\mathcal{O}_2)$\\
$1 \notin S$: & $1{+}1{+}1{+}1{+}1$ & $2^5 = 32$            & $(=\mathcal{O}_3)$\\
              & $2{+}1{+}1{+}1$     & $5\binom{4}{3}2^3 = 160$ & $(=\mathcal{O}_4)$\\
              & $2{+}2{+}1$         & $\binom{5}{2}\cdot3\cdot2 = 60$ & $(=\mathcal{O}_5)$\\
\bottomrule
\end{tabular}
\end{center}
The sizes sum to $210 + 252 = 462$, and
$\mathcal{O}_0 \cup \mathcal{O}_1 \cup \mathcal{O}_2$ is exactly the set of
$\binom{10}{4} = 210$ blocks containing the point $1$.
\end{proposition}

\begin{proof}
The type is $\gp$-invariant, since $\gp$ fixes the point $1$ and permutes the pairs. Conversely
$\gp$ is transitive on each type: given two $5$-sets of the same type, map pairs to pairs matching
intersection sizes, then swap within pairs to match the chosen points. So the orbits are exactly
the types. The counts are elementary: e.g.\ for type $2{+}1{+}1$ through the point $1$, choose the
full pair ($5$ ways), the two half-met pairs ($\binom{4}{2}$), and one point in each ($2^2$).
The remaining rows are identical in kind, and the two impossible patterns ($2{+}2{+}1$ through
$1$ would need five points besides $1$; $1^4$ avoiding $1$ needs a fifth point) do not occur.
\end{proof}

For a candidate link $\link$, define its \emph{root index} $r$ as the least $i$ such that $\link$
contains a block of $\mathcal{O}_i$. Since $b(1) = 10 > 0$, every candidate link contains blocks
through the point $1$, all of which lie in $\mathcal{O}_0 \cup \mathcal{O}_1 \cup \mathcal{O}_2$;
hence $r \in \{0,1,2\}$, and the three cases $r=0$, $r=1$, $r \ge 2$ are exhaustive and mutually
exclusive. (The instance closing the case $r \ge 2$ in Section~\ref{sec:cases} does not use the
bound $r \le 2$, so the analysis remains complete even without this observation; a separate
blocker-free certificate, called the $r\ge 3$ auxiliary instance in the deposit, confirms directly
that no candidate link avoids all blocks through the point $1$.)

\section{The two families of satisfiability instances}\label{sec:encoding}

The computation has two layers. \emph{Link instances} (Layer A) decide, case by case, whether a
candidate link with prescribed properties exists. \emph{Extension instances} (Layer B) certify, for
individual optimal $3$-$(11,5,1)$ coverings, that they do not extend to a $40$-block covering.
Layer B feeds Layer A through the blocker clauses described below.

\subsection{Link instances}\label{ssec:linkcnf}

\paragraph{Variables.} One Boolean variable $x_B$ for each of the $462$ five-subsets $B$ of
$\{1,\dots,11\}$, reading $x_B$ as $B \in \link$.

\paragraph{Coverage.} For each of the $165$ triples $S$, the clause
$\bigvee_{B \supseteq S} x_B$, a disjunction over the $\binom{8}{2} = 28$ blocks containing $S$:
this family of clauses says exactly that $\link$ is a $3$-$(11,5,1)$ covering.

\paragraph{Degrees.} By Lemmas~\ref{lem:degrees} and~\ref{lem:forcing} and our choice of
labelling, the link satisfies $b(1) = 10$ and $b(q) = 9$ for $q = 2,\dots,11$. Each of these eleven
constraints is an exact cardinality constraint over the $\binom{10}{4} = 210$ variables $x_B$ with
$q \in B$. Together they imply $|\link| = 20$, since $\sum_q b(q) = 100 = 5\,|\link|$; no separate
global cardinality constraint is needed. Cardinality constraints are not natively clausal, so the whole computation
was carried out under two structurally different encodings, a sequential counter~\cite{Sinz2005} and a
totalizer-based encoding, both as supplied by PySAT~\cite{PySAT}, producing instances with $40{,}642$ and $20{,}482$
variables respectively. Both encoders were validated against ground truth before use: for
fourteen small $(n,k)$ cases, the models of the encoded constraint were enumerated over all $2^n$
assignments to the input literals and confirmed to be exactly the weight-$k$ assignments.

\paragraph{Case constraints.} A case of the analysis is imposed by unit clauses: positive units
fixing canonical representative blocks to be present, and negative units forbidding the blocks of
excluded orbits. Section~\ref{sec:cases} specifies the cases and proves their exhaustiveness.

\paragraph{Blocker clauses.} Purely negative clauses of width exactly $20$,
\begin{equation}\label{eq:blockerclause}
  \neg x_{B_1} \vee \neg x_{B_2} \vee \cdots \vee \neg x_{B_{20}} ,
\end{equation}
one for each member of a library of optimal $3$-$(11,5,1)$ coverings certified not to extend
(Section~\ref{ssec:blocker}). Since the link has exactly $20$ blocks, the clause
\eqref{eq:blockerclause} excludes exactly one candidate link, namely
$\{B_1,\dots,B_{20}\}$ itself.

\subsection{Extension instances}\label{ssec:extension}

Let $\link$ be a fixed optimal $3$-$(11,5,1)$ covering with degree sequence normalised as in
Section~\ref{sec:structure}. The instance $E(\link)$ asks whether $\link$ extends to a $40$-block
covering of $\pts = \{1,\dots,11\} \cup \{p_0\}$. If it does, the $20$ blocks through $p_0$ are
exactly $\{B \cup \{p_0\} : B \in \link\}$, and the remaining $20$ blocks --- the
\emph{co-blocks} --- are $6$-subsets of $\{1,\dots,11\}$. The co-blocks must cover every
$4$-subset $Q$ of $\{1,\dots,11\}$ that is not already covered by $\link$ (i.e., with
$Q \not\subseteq B$ for all $B \in \link$; a quadruple containing $p_0$ is covered through
$\link$ automatically, because $\link$ covers triples).

The remaining constraints encode the pair degrees forced by Lemma~\ref{lem:matching}. Define
\[
 d(qr)=
 \begin{cases}
 10,&\{q,r\}\in\bigl\{\{2,3\},\{4,5\},\{6,7\},\{8,9\},\{10,11\}\bigr\},\\
 9,&\text{otherwise},
 \end{cases}
\]
for pairs $\{q,r\}\subseteq\{1,\dots,11\}$. If $b_\link(qr)$ is the number of root-link blocks
containing both $q$ and $r$, then exactly
\[
 d(qr)-b_\link(qr)
\]
co-blocks must contain $\{q,r\}$. These are the residual pair-degree equations after the perfect
matching has been normalised as in Section~\ref{sec:structure}. Summing the $55$ residual
equations gives $5\cdot 10 + 50\cdot 9 - 20\binom{5}{2} = 300$ residual pair incidences; since
each selected co-block contributes $\binom{6}{2}=15$ of them, every satisfying assignment of
$E(\link)$ selects exactly $300/15 = 20$ co-blocks.

Accordingly, $E(\link)$ has one variable per $6$-subset of $\{1,\dots,11\}$
($\binom{11}{6}=462$ of them), one width-$21$ clause per uncovered quadruple, and
$\binom{11}{2}=55$ exact cardinality constraints, one for each pair of link points, fixing the
residual value $d(qr)-b_\link(qr)$. The link itself is substituted into these bounds rather than
asserted by unit clauses. Every $40$-block covering with link $\link$ at $p_0$ therefore gives a
satisfying assignment of $E(\link)$.

\begin{lemma}\label{lem:exttest}
If $E(\link)$ is unsatisfiable, then no $4$-$(12,6,1)$ covering with $40$ blocks has $\link$ as
the root link under the perfect-matching normalisation of Section~\ref{sec:structure}.
\end{lemma}

\begin{proof}
A $40$-block covering whose normalised root link is $\link$ yields a satisfying assignment of
$E(\link)$: its co-blocks satisfy the coverage clauses by the covering property and the $55$
residual equations by Lemma~\ref{lem:matching} and the definition of $d(qr)$.
\end{proof}

\begin{lemma}\label{lem:transfer}
For every $\sigma \in \gp$, the instances $E(\link)$ and $E(\sigma \link)$ are isomorphic:
relabelling variables by $\sigma$ maps one onto the other. In particular, one unsatisfiability
certificate per $\gp$-orbit suffices for the whole orbit.
\end{lemma}

\begin{proof}
$\sigma$ induces a bijection of $6$-subsets carrying the uncovered quadruples of $\link$ to those
of $\sigma\link$. Because $\gp$ preserves the five matched pairs, it also carries each residual
pair-degree equation of $E(\link)$ to the corresponding equation of $E(\sigma\link)$.
\end{proof}

\subsection{The blocker}\label{ssec:blocker}

The \emph{blocker} is a library of $15{,}120$ optimal $3$-$(11,5,1)$ coverings, each certified not
to extend, contributing $15{,}120$ clauses of the form \eqref{eq:blockerclause} to Layer A. It is
closed under $\gp$, being a union of exactly $20$ complete $\gp$-orbits of sizes
\[
  1\times 16,\quad 2\times 160,\quad 2\times 240,\quad 3\times 320,\quad 1\times 384,\quad
  3\times 480,\quad 4\times 960,\quad 4\times 1920
\]
(multiplicity $\times$ orbit size), summing to $15{,}120$. Closure was verified against the
three-element generating set of $\gp$ recorded in the archived deposit, and the set was verified
to be a union of complete orbits rather than of partial ones. By Lemma~\ref{lem:transfer}, $20$ extension certificates --- one per orbit --- cover
all $15{,}120$ members, and all $20$ are archived and verified.

The blocker was built incrementally. Nine orbits were certified first; under the resulting
$9$-orbit blocker, $43$ of the $47$ case instances
of Section~\ref{sec:cases} were already unsatisfiable. Each of the remaining four was satisfiable,
and each satisfying assignment decoded to an optimal covering not yet in the library; its
non-extendability was certified via $E(\cdot)$, its full $\gp$-orbit adjoined, and the case
re-solved. Three of the four cases closed after one such round (a $13$-orbit blocker); the last,
the case instance \texttt{s-r0-2}, required seven further rounds, ending at the final $20$-orbit
blocker. Consequently $43$ case instances carry $9$ blocker orbits, three carry $13$, and one
carries all $20$.

We emphasise the direction of the logic: the blocker is used only \emph{soundly}. Each blocker
clause is individually backed by an extension certificate, and no completeness property of the
blocker is ever assumed --- the exhaustiveness of the case analysis comes from the root-index split
and the degree constraints, never from the blocker. The correctness of each clause thus rests on
the $E(\cdot)$ encodings, which are covered by the independent byte-level encoding audit of
Section~\ref{sec:verification}.

\paragraph{Hash binding.} Every instance, certificate, and log is bound to the claims of this paper
by its SHA-256 digest, recorded in the manifest of the archived deposit
(Section~\ref{sec:verification}).

\section{The case analysis}\label{sec:cases}

Throughout, ``closed'' means that the corresponding instance was found unsatisfiable and its
certificate independently checked as described in Section~\ref{sec:verification}. All case
instances contain the coverage and degree clauses of Section~\ref{ssec:linkcnf} and, in addition,
a subset of the blocker clauses; the branching structure is justified once and for all by the following lemma.

\begin{lemma}[orbit branching]\label{lem:branch}
Let $H \le S_{11}$, and let $\mathcal{C}$ be a set of constraints on candidate links that is
$H$-invariant ($\link$ satisfies $\mathcal{C}$ iff $h\link$ does, for all $h \in H$). Let
$Q_1,\dots,Q_m$ be $H$-orbits of blocks with representatives $R_1,\dots,R_m$. Suppose that
\begin{itemize}
\item[(i)] for each $i$, no link satisfying $\mathcal{C}$ contains $R_i$ and avoids
  $Q_1 \cup \dots \cup Q_{i-1}$; and
\item[(ii)] no link satisfying $\mathcal{C}$ avoids $Q_1 \cup \dots \cup Q_m$.
\end{itemize}
Then no link satisfies $\mathcal{C}$.
\end{lemma}

\begin{proof}
Suppose $\link$ satisfies $\mathcal{C}$. By (ii), $\link$ meets some $Q_i$; choose $i$ least, and
$b \in \link \cap Q_i$. By transitivity of $H$ on $Q_i$ there is $h \in H$ with $h(b) = R_i$. Then
$h\link$ satisfies $\mathcal{C}$ (invariance), contains $R_i$, and avoids
$Q_1 \cup \dots \cup Q_{i-1}$: indeed $\link$ avoids these orbits by minimality of $i$, and each
$Q_j$ is $H$-invariant. This contradicts (i).
\end{proof}

Two remarks on how Lemma~\ref{lem:branch} is applied. First, a hypothesis of type (i) may itself
be established by a nested application, with $H$ replaced by the stabiliser of $R_i$ and
$\mathcal{C}$ extended by the (stabiliser-invariant) conditions ``$R_i \in \link$, no earlier orbit
met''. Second, an instance certifying (i) or (ii) may contain only a subset of the clauses of
$\mathcal{C}$ (in practice, a sub-blocker): unsatisfiability with fewer constraints implies
unsatisfiability with all of them. Below, $\mathcal{C}$ is always
\emph{coverage $\wedge$ degrees $\wedge$ blocker}, which is $\gp$-invariant: coverage and the
blocker by construction, the degrees because $\gp$ fixes the point $1$.

At the root, Lemma~\ref{lem:branch} is applied with $H = \gp$ and the orbits
$\mathcal{O}_0, \mathcal{O}_1$: case (i) for $\mathcal{O}_0$ is the region $r=0$, case (i) for
$\mathcal{O}_1$ is the region $r=1$, and case (ii) is the region $r \ge 2$.

\subsection{The region \texorpdfstring{$r \ge 2$}{r >= 2}}\label{ssec:rge2}

A single instance: coverage, degrees, the full $20$-orbit blocker, and $200$ negative units
forbidding every block of $\mathcal{O}_0 \cup \mathcal{O}_1$. It is unsatisfiable, and no
canonicity assumption is involved, so the conclusion applies to every candidate link avoiding
$\mathcal{O}_0$ and $\mathcal{O}_1$: hypothesis (ii) of the root application.

\begin{remark}\label{rem:repair}
The blocker clauses are essential here: the same instance \emph{without} them is satisfiable, and
its model decodes to an explicit optimal covering $E$ with root index $2$, blocks drawn from
$\mathcal{O}_2$ and $\mathcal{O}_3$, and degree sequence $(10,9^{10})$ --- as forced by
Lemma~\ref{lem:degrees}. $E$ is the witness of Proposition~\ref{prop:c1153}; it lies in the
blocker orbit of size $16$ --- the smallest of the twenty, with stabiliser of order $240$ in $\gp$
--- and its non-extendability is certified by that orbit's extension certificate.
\end{remark}

\subsection{The region \texorpdfstring{$r = 0$}{r = 0}}\label{ssec:r0}

Fix the canonical representative $\{1,2,4,6,8\}$ of $\mathcal{O}_0$; its stabiliser in $\gp$ has
order $48$, consistently with $3840/48 = 80 = |\mathcal{O}_0|$. The nested application of
Lemma~\ref{lem:branch} uses this stabiliser acting on the \emph{legal pool} --- the blocks still
available once $\{1,2,4,6,8\}$ is in the link --- which decomposes into $39$ secondary orbits.
Of these:
\begin{itemize}
  \item $6$ are closed by individual instances (hypotheses of type (i));
  \item $1$, with representative $\{1,2,3,4,5\}$, is expanded into a third level of case analysis,
        treated in Section~\ref{ssec:r0s0};
  \item the remaining $32$ are covered by a single tail instance (hypothesis (ii)): coverage,
        degrees, the root representative fixed, and negative units forbidding the seven named
        orbits entirely. It is unsatisfiable with no blocker clauses at all.
\end{itemize}
Since $6+1+32 = 39$, every secondary orbit is accounted for.

\subsection{The third level under \texorpdfstring{$r=0$}{r=0}}\label{ssec:r0s0}

With $\{1,2,4,6,8\}$ and $\{1,2,3,4,5\}$ both fixed, the joint stabiliser has order $8$, and the
legal pool decomposes into $122$ tertiary orbits: $33$ closed by individual instances and the
remaining $89$ by a single blocker-free tail instance; $33+89 = 122$.

\subsection{The region \texorpdfstring{$r = 1$}{r = 1}}\label{ssec:r1}

Fix the canonical representative $\{1,2,3,4,6\}$ of $\mathcal{O}_1$, whose stabiliser has order
$32$; again $3840/32 = 120 = |\mathcal{O}_1|$. Here the case hypothesis excludes all of
$\mathcal{O}_0$, and on the resulting legal pool the stabiliser has exactly $60$ secondary orbits:
$12$ closed by individual instances and the remaining $48$ by a single blocker-free tail instance;
$12+48 = 60$.

\begin{remark}\label{rem:overcount}
The archived deposit closes this region with sixteen individual instances rather than twelve: they
were enumerated over the full set of remaining blocks rather than over the legal pool, so they
include configurations that reintroduce a block of $\mathcal{O}_0$ and are already excluded by the
case hypothesis. The sixteen collapse onto the $12$ distinct legal orbits, the other four being
duplicates with byte-identical trimmed proofs. The deposit therefore decides four redundant
branches beyond the minimal case space --- redundant, never unsound.
\end{remark}

\subsection{Completeness checks}\label{ssec:completeness}

Two properties were checked by programs written independently of the code that generated the
instances. \emph{Region completeness}: in each region, the orbits whose blocks a tail instance
forbids wholly (the ``named'' orbits) all have certified closures --- $7/7$
for $r=0$, $33/33$ for the third level, $12/12$ for $r=1$ --- with no orbit missing. Each of these
closures is a single solved instance, with the one exception of the orbit of $\{1,2,3,4,5\}$ under
$r=0$, which is closed by the third-level analysis of Section~\ref{ssec:r0s0}.
\emph{Tail soundness}: no tail instance forbids part of an orbit while leaving the rest
unaccounted for; the number of straddling orbits is $0$ in all three regions. Together these
confirm that the hypotheses of Lemma~\ref{lem:branch} are certified at every level.

\subsection{Proof of the main theorem}

\begin{proof}[Proof of Theorem~\ref{thm:main}]
The archived $41$-block design covers all $\binom{12}{4} = 495$ quadruples
(Section~\ref{sec:verification}), so $\Cov{12}{6}{4} \le 41$.

Suppose a $4$-$(12,6,1)$ covering with $40$ blocks exists. Normalise its perfect matching as in
Section~\ref{sec:structure}, and let $\link$ be the link of the root point $p_0$. By
Lemma~\ref{lem:forcing}, $\link$ satisfies the coverage and degree constraints. The case analysis
of Sections~\ref{ssec:rge2}--\ref{ssec:completeness} establishes, via Lemma~\ref{lem:branch}
applied with $\mathcal{C} = \text{coverage} \wedge \text{degrees} \wedge \text{blocker}$, that no
candidate link satisfies $\mathcal{C}$. Hence $\link$ violates some blocker clause
\eqref{eq:blockerclause}; since $|\link| = 20$, this means $\link$ \emph{equals} the blocked
covering of that clause. But every blocker member lies in one of the $20$ certified
$\gp$-orbits, so by Lemma~\ref{lem:transfer} the instance $E(\link)$ is unsatisfiable, and by
Lemma~\ref{lem:exttest} no matching-normalised $40$-block covering has $\link$ as its root link
--- contradicting the fact that $\link$ is the root link of the covering just normalised. So no
$40$-block covering exists, and $\Cov{12}{6}{4} = 41$.
\end{proof}

\subsection{Uniqueness of the optimal \texorpdfstring{$3$-$(11,5,1)$}{3-(11,5,1)} covering}

The same certificates decide a classical uniqueness question. The blocked set is closed under
$\gp$ but not under $S_{11}$; a short breadth-first orbit computation, included in the deposit,
shows that the $S_{11}$-orbit of the witness $E$ has size
$166{,}320 = 11!/240$ and contains all $15{,}120$ blocked coverings. The two numbers corroborate
each other: $166{,}320/11 = 15{,}120$ is exactly the number of orbit members whose degree-$10$
point is the point $1$.

\begin{proposition}\label{prop:unique}
Up to permutations of the point set there is exactly one $3$-$(11,5,1)$ covering with $20$ blocks.
Its automorphism group has order $240$, so there are $166{,}320$ labelled optimal coverings.
\end{proposition}

\begin{proof}
Existence is Proposition~\ref{prop:c1153}. Let $M$ be any optimal covering. By
Lemma~\ref{lem:degrees} it has a unique degree-$10$ point; relabel so that this point is the
point $1$. Then $M$ satisfies the coverage and degree constraints, and by the case analysis (as in
the proof of Theorem~\ref{thm:main}) it violates a blocker clause, i.e., $M$ is one of the
$15{,}120$ blocked coverings. All of these lie in a single $S_{11}$-orbit --- the orbit of $E$,
by the computation above --- so $M \cong E$. The orbit size $166{,}320 = 11!/240$ gives
$|{\operatorname{Aut}}(E)| = 240$ by orbit--stabiliser.
\end{proof}

The proof above establishes Proposition~\ref{prop:unique} from the certified case analysis alone.
This uniqueness result is not new: van Rees~\cite{vanRees1994} proves both $\Cov{11}{5}{3}=20$
and the uniqueness of the $20$-block covering up to isomorphism.  The contribution here
is a reproducible, certificate-backed reproof and an independently checkable automorphism-order
calculation.
Note that Proposition~\ref{prop:unique} uses only Layer A: the blocker clauses enter as
\emph{constraints}, and the argument never needs the fact that their members fail to extend.

\section{Verification}\label{sec:verification}

This section records how each computational claim of the paper is checked --- the upper-bound
design directly, every unsatisfiability claim through a certificate pipeline ending in a formally
verified checker, and the encodings themselves through an independent byte-level audit --- and
closes with a precise statement of the residual trust assumptions.

\paragraph{Upper bound.} The $41$-block design was checked by two independently written programs,
one in Python using bitmask representations and one in C using nested-loop membership tests,
applied to two separately obtained copies of the design (the stored copy and a fresh retrieval
from the repository, byte-identical after normalisation). Both programs confirm $41$ distinct
blocks of six distinct points each and all $495$ quadruples covered; both reject a negative
control obtained by corrupting a single element.

\paragraph{Solver and proof checking.} All instances were solved with
CaDiCaL~3.0.1~\cite{CaDiCaL}. The solver is not trusted: each unsatisfiability claim is emitted as
a DRAT certificate, checked by \texttt{drat-trim}~\cite{DRATtrim} (required to report
\texttt{s VERIFIED}), converted to LRAT, and re-checked by \texttt{cake\_lpr}~\cite{cakelpr}, a
checker whose correctness is itself a machine-checked theorem in the CakeML and HOL4
ecosystem. The reproduction pipeline, per instance, is:
\begin{center}\small
\texttt{cadical F.cnf F.drat} $\to$ \texttt{drat-trim F.cnf F.drat -L F.lrat} $\to$
\texttt{cake\_lpr F.cnf F.lrat}
\end{center}
Exact tool versions, sources and build instructions, together with exact pins of every Python
dependency of the instance generators, are recorded in the archived deposit.

\paragraph{Certificate inventory.} The primary proof archive contains $81$ certificates. Two
additional case instances under the second cardinality encoding bring the checked inventory to
$83$. Every certificate was accepted by \texttt{drat-trim} and \texttt{cake\_lpr}, with no
failures, hash mismatches or empty proofs:

\begin{center}
\begin{tabular}{lrl}
\toprule
family & count & establishes \\
\midrule
case instances (\S\ref{sec:cases}) & $47$ & the $6+8+33$ individual case closures \\
auxiliary case instances & $14$ & the $r\ge 2$ and $r\ge 3$ instances, the three tails, \\
 & & the eight additional $r{=}1$ instances, and Lemma~\ref{lem:c1042} \\
extension instances (\S\ref{ssec:extension}) & $20$ & one per blocker orbit \\
\midrule
primary archive & $81$ & \\
second-encoding additions & $2$ & two case instances re-proved \\
\midrule
checked inventory & $83$ & \\
\bottomrule
\end{tabular}
\end{center}

The $r\ge 3$ label refers to the blocker-free certificate described in
Section~\ref{sec:symmetry}. The split between the case and auxiliary families for $r=1$ is
archival: eight of the sixteen jobs are counted in each family; together they represent twelve
legal-orbit closures and four duplicate certificates, as explained in
Remark~\ref{rem:overcount}.

Of the $47$ primary case certificates, $46$ check sequential-counter instances. The largest case,
\texttt{s-r0-2}, checks the totalizer instance; its sequential counterpart was also solved and
accepted by \texttt{drat-trim}, but that large proof was not retained. The manifest records both
CNF hashes and their identical non-cardinality core.

\paragraph{Cross-encoding replication.} All $47$ case instances were regenerated under the second
cardinality encoding and re-solved from scratch: the regenerated instances agree byte-for-byte
with the originals on their non-cardinality core in all $47$ cases while differing genuinely in
cardinality structure, and all $47$ are unsatisfiable under the second encoding with
\texttt{drat-trim}-verified proofs. Proofs were retained for $46$ cases; for
\texttt{s-r0-2}, the sequential cross-encoding proof was checked and then discarded, while its
certified totalizer proof remains in the primary archive. Two additional instances carry
\texttt{cake\_lpr} certificates, counted above. There were no satisfiable results and no timeouts.
We describe this as a cross-encoding check: both translations encode the same non-cardinality core
and use structurally different cardinality clauses.

\paragraph{Compute.} The second-encoding sweep required $12{,}347$~s of solving and $10{,}310$~s
of proof checking over $16.3$~GiB of DRAT, on an eight-core machine with $8$~GB of memory. The
largest single instance produced $3.08$~GiB of certificate ($2826$~s to solve, $1710$~s to check).

\paragraph{Encoding audit.} Three auditors, written from first principles in pure
standard-library Python with no dependency on the encoding library or on any code that produced
the instances, reconstruct the complete $81$-formula sequential baseline --- the $47$ case
formulas, the $20$ extension formulas, and the $14$ auxiliary formulas --- directly from their
mathematical descriptions: coverage clauses from the combinatorics, cardinality
segments from a clean-room implementation of the pruned sequential-counter construction of
Sinz~\cite{Sinz2005}, blocker prefixes re-parsed from their pinned files, and case tails checked
to consist of clauses over primary variables only. In all $81$ cases the reconstruction
reproduces the shipped baseline instance \emph{byte-for-byte}, with SHA-256 equal to the manifest
pins. Thus every primary formula except the totalizer version of \texttt{s-r0-2} is itself
byte-audited; for \texttt{s-r0-2}, the audited sequential formula has the same non-cardinality
core as the certified totalizer formula, and both cardinality translations pass the exhaustive
small-instance validation described in Section~\ref{ssec:linkcnf}.

\paragraph{Verification scope.} The hand argument reduces the theorem to the certified finite
instances above. Separate programs recompute the group, orbit partitions, blocker closure,
branch completeness, pair-degree extension formulas and witness identities. Every
unsatisfiability claim used in the primary proof has a deposited certificate checked by both
\texttt{drat-trim} and the formally verified checker \texttt{cake\_lpr}; the second-encoding
sweep provides additional corroboration, with one checked cross-encoding proof not retained while
the corresponding certified primary proof remains deposited. The remaining trusted surface is the short translation and audit code,
the documented cardinality constructions, the CakeML/HOL4 checker stack and the computing
platform. This is the standard verification model for large SAT-based combinatorial
proofs~\cite{PythagoreanTriples,SchurFive,KellerConjecture}.

\section{Concluding remarks}\label{sec:concluding}

\subsection{The Tur\'an reformulation}

For $n \ge \ell \ge k$, the \emph{Tur\'an number} $T(n,\ell,k)$ is the least number of
$k$-subsets of an $n$-set such that every $\ell$-subset contains at least one of them; it is the
complement-dual of the covering number, $T(n,\ell,k) = \Cov{n}{n-k}{n-\ell}$
\cite[Cor.~1.9]{HandbookCoverings}. Writing $K^{(k)}_\ell$ for the complete $k$-uniform
hypergraph on $\ell$ vertices and $\ex(n; K^{(k)}_\ell)$ for the maximum number of edges of a
$k$-uniform hypergraph on $n$ vertices containing no copy of $K^{(k)}_\ell$, one has
$\ex(n; K^{(k)}_\ell) = \binom{n}{k} - T(n,\ell,k)$.

\begin{corollary}\label{cor:turan}
$T(12,8,6) = 41$; equivalently, $\ex(12; K^{(6)}_8) = 924 - 41 = 883$.
\end{corollary}

We stress that Corollary~\ref{cor:turan} is an exact finite value and carries no asymptotic
content: the density $883/924$ it yields at $n=12$ is weaker than what follows from de Caen's
general bounds on Tur\'an numbers~\cite{deCaen1983} (see also the survey by
Keevash~\cite{Keevash2011}), so it gives no new information about the Tur\'an density of
$K^{(6)}_8$.

\subsection{Consequences for other covering numbers}

Applying \eqref{eq:schonheim} repeatedly, the value $\Cov{12}{6}{4}=41$ propagates along the
family of parameter triples with $v-k = 6$ and $v-t = 8$. Every previously tabulated lower bound
on this family is reproduced exactly by \eqref{eq:schonheim} seeded with the old value $40$
($\lceil 13\cdot 40/7\rceil = 75$, $\lceil 14\cdot 75/8\rceil = 132$,
$\lceil 15\cdot 132/9\rceil = 220$, $\lceil 16\cdot 220/10\rceil = 352$), so no stronger bound was
in force and each entry improves.

\begin{corollary}\label{cor:propagation}
$\Cov{13}{7}{5} \ge 77$, $\Cov{14}{8}{6} \ge 135$, $\Cov{15}{9}{7} \ge 225$ and
$\Cov{16}{10}{8} \ge 360$.
\end{corollary}

\begin{proof}
By \eqref{eq:schonheim} and Theorem~\ref{thm:main}:
$\lceil \tfrac{13}{7}\cdot 41\rceil = 77$, then $\lceil \tfrac{14}{8}\cdot 77\rceil = 135$,
$\lceil \tfrac{15}{9}\cdot 135\rceil = 225$, $\lceil \tfrac{16}{10}\cdot 225\rceil = 360$.
\end{proof}

\begin{table}[ht]
\centering
\begin{tabular}{lccccc}
\toprule
 & previous & new & gain & best known upper & gap \\
\midrule
$\Cov{12}{6}{4}$  & $40$  & $\mathbf{41}$  & $+1$ & $41$  & $1 \to 0$ \\
$\Cov{13}{7}{5}$  & $75$  & $\mathbf{77}$  & $+2$ & $78$  & $3 \to 1$ \\
$\Cov{14}{8}{6}$  & $132$ & $\mathbf{135}$ & $+3$ & $151$ & $19 \to 16$ \\
$\Cov{15}{9}{7}$  & $220$ & $\mathbf{225}$ & $+5$ & $270$ & $50 \to 45$ \\
$\Cov{16}{10}{8}$ & $352$ & $\mathbf{360}$ & $+8$ & $448$ & $96 \to 88$ \\
\bottomrule
\end{tabular}
\caption{Lower bounds improved by Theorem~\ref{thm:main}. Previous values and upper bounds are
those recorded in \cite{LJCR,CoveringRepository}, as pinned in the repository snapshot archived
with the deposit. The family continues beyond the tabulated range; for instance
$\Cov{17}{11}{9} \ge 557$.}
\label{tab:prop}
\end{table}

In the notation of Corollary~\ref{cor:turan} the family is a single sequence of Tur\'an
numbers, $\Cov{v}{v-6}{v-8} = T(v,8,6)$ for $v \ge 12$: Theorem~\ref{thm:main} determines its
initial term exactly and raises the best known lower bound for every later term, e.g.\ $T(13,8,6)\ge77$ and
$T(14,8,6)\ge135$.

\subsection{Outlook}

The feature that made these parameters tractable is the rigidity of Lemma~\ref{lem:forcing}: the
bound \eqref{eq:known} is attained with zero slack, forcing the link of every point to be an
optimal covering with forced degrees and collapsing the search space accordingly. This occurs
whenever $v\,\Cov{v-1}{k-1}{t-1}$ is divisible by $k$ and the resulting bound is one short of the
true value --- a configuration in which the divisibility hypothesis of the Mills--Mullin
increment fails. The
natural next target, $\Cov{13}{7}{5}$, whose gap Corollary~\ref{cor:propagation} reduces to one,
is \emph{not} of this kind: $13 \cdot 41 = 533$ while $7 \cdot 77 = 539$, so a hypothetical
$77$-block covering has six units of degree slack, and the link of a point need not be optimal.
Deciding it would need either the present method scaled to a weaker forcing regime or a new idea;
parameter families where the lifted bound is again tight are the better candidates for a direct
replay of this approach.

\subsection*{Use of AI assistance}

During the preparation of this work the author used the large language models GPT-5.6~Sol,
GPT-5.6~Terra and Claude Fable~5, in the versions available in July~2026, to assist with
exploratory analysis, computational search, supporting code, manuscript drafting and revision,
and proofreading. The verification status and residual trust assumptions of the reported
computations are stated in Section~\ref{sec:verification}. The author reviewed and edited all
model output and takes full responsibility for the content of this paper.

\subsection*{Data availability}

The source code, all exact CNF instances checked by the retained certificates, verification
programs, logs, the pinned repository snapshot, and the uniqueness-orbit computation supporting
this paper are archived in version~1.0.1 of
\url{https://doi.org/10.5281/zenodo.21572069}; that DOI represents all versions of the
artifact record and resolves to the latest one. The $81$ DRAT unsatisfiability certificates
refuting those instances are archived as a companion record at
\url{https://doi.org/10.5281/zenodo.21573716}. SHA-256 digests of
all files are recorded in the deposit manifests and serve as their canonical identifiers.

Two derived sets are not deposited. The $47$ primary-case LRAT files, about $13$~GiB, regenerate
deterministically from the deposited CNF and DRAT pairs via \texttt{drat-trim} \texttt{<cnf>}
\texttt{<drat>} \texttt{-L}; and the retained second-encoding DRAT proofs for $46$ of the $47$
frontier nodes, about $13$~GiB, regenerate from the archived instances via the pipeline of
Section~\ref{sec:verification}. Both are available from the author on request. Every object in
either set is identified by its SHA-256 digest and byte length in \texttt{data/frontier.json}, so
a re-derivation can be compared object by object without transferring them.

\end{document}